\documentclass{amsart}

\usepackage{latexsym,amsmath,amsfonts,amscd,amssymb, amsthm,url}

\theoremstyle{plain}
\newtheorem{theorem}{Theorem}[section]
\newtheorem{lemma}{Lemma}[section]
\newtheorem{proposition}{Proposition}[section]
\newtheorem{corollary}{Corollary}[section]

\theoremstyle{definition}
\newtheorem{definition}{Definition}[section]
\newtheorem{example}{Example}[section]

\theoremstyle{remark}
\newtheorem{remark}{Remark}[section]

\title{Homotopy equivalence for proper holomorphic mappings}

\author{John P. D'Angelo}

\address{Dept. of Mathematics, Univ. of Illinois, 1409 W. Green St.,
Urbana IL 61801}

\email{jpda@math.uiuc.edu}

\author{Ji\v{r}\'\i\ Lebl}

\address{Dept. of Mathematics, Oklahoma State Univ., Stillwater OK 74078}

\email{lebl@math.okstate.edu}

\begin{document}

\begin{abstract} We introduce several homotopy equivalence relations for proper holomorphic
mappings between balls. We provide examples showing that the degree of a rational proper mapping
between balls (in positive codimension) is not a homotopy invariant. In domain dimension at least $2$,
we prove that the set of homotopy classes of rational proper mappings from a ball to a higher dimensional ball
is finite. By contrast, when the target dimension is at least twice the domain dimension, it is well
known that there are uncountably many spherical
equivalence classes. We generalize this result by proving that an arbitrary 
homotopy of rational maps whose endpoints are spherically
inequivalent must contain uncountably many spherically inequivalent maps.
We introduce Whitney sequences, a precise analogue (in higher dimensions)
of the notion of finite Blaschke product (in one dimension). We
show that terms in a Whitney sequence are homotopic to monomial mappings,
and we establish an additional result about the target dimensions of such homotopies.

\medskip

\noindent
{\bf AMS Classification Numbers}: 32H35, 32H02, 32M99, 32A50, 55P10, 30J10, 14P10.

\medskip

\end{abstract}

\maketitle

\section{Introduction}

This paper considers proper holomorphic mappings between balls in
possibly different dimensional complex Euclidean spaces.
Two notions of equivalence (spherical and norm) for such maps have
been extensively used. See for example
[Fa], [F2], [H], [HJ], [L], [D], [D3], [R], [DL2].  The purpose of this paper is
to introduce and investigate
a natural but subtle third notion, homotopy equivalence, which is more
useful for some purposes.
Homotopy equivalence itself has several possible definitions, each of
which is useful in different contexts.

The one-dimensional situation for homotopy equivalence is precise,
beautiful, and easy to describe.
See Proposition 2.1. It is natural to attempt to generalize that
result to higher dimensions.
Theorem 5.1 provides one precise analogue of Proposition 2.1.
Several crucial differences arise, however, which we confront in this paper.

Let ${\bf C}^n$ denote complex Euclidean space and let ${\bf B}_n$
denote the unit ball in ${\bf C}^n$.
A holomorphic map $f:{\bf B}_n \to {\bf B}_N$ is {\it proper} if and
only if, for each compact subset $K$ of the target ball,
the inverse image $f^{-1}(K)$ is compact in the domain ball.

The basic properties of homotopy developed in this paper do not depend on regularity assumptions
of the mappings at the boundary. Definitions 2.1, 2.2, and 2.3 introduce the various notions of homotopy. 
One of the key issues, motivating Definition 2.3, allows a homotopy between maps
whose target dimensions differ.

Proposition 2.4 shows that any pair of proper maps from the same ball
are {\it homotopy equivalent in target dimension $M$} when $M$ is 
sufficiently large.
Placing restrictions on $M$ then fits nicely into the general
philosophy of complexity theory in CR Geometry.
In particular, given proper maps $f$ and $g$ with the same domain
ball, there is a minimal $M$ for which
$f$ and $g$ are homotopy equivalent in target dimension $M$. Computing 
this dimension
for explicit rational maps seems to be difficult.

When the domain dimension is at least $2$,
a proper mapping between balls that is smooth up to the boundary must be, by a well-known theorem of Forstneric [F1], 
a rational mapping. It is therefore important also to consider homotopies 
where all the maps in the family are rational (Definition 2.2).

Example 2.1 is striking; it
 shows that the degree of a family of rational 
proper mappings between balls is {\it not} a homotopy invariant.
Theorem 5.1 provides large classes of homotopic proper rational maps 
(terms of Whitney sequences)
for which the degree need
not be a homotopy invariant and clarifies Example 2.1. The proof of this 
result illuminates
a fundamental distinction between the one-dimensional case and the general case. A finite Blaschke product of degree $d$
is the $d$-th term of a Whitney sequence; in domain dimension at least two, however, the $d$-th term of a Whitney sequence
can be of degree less than $d$. Furthermore,
in the higher dimensional case, there exist rational proper mappings between balls that are not terms of Whitney sequences.

Theorem 3.1 gives a finiteness result:
for $n\ge 2$ and $N$ fixed, the set of homotopy classes of rational
proper maps from ${\bf B}_n$ to ${\bf B}_N$ is finite.
Theorem 3.1 is useful, because, by contrast, the number of distinct
spherical equivalence classes is infinite in general. 
Theorem 3.2 and Corollary 3.1 
decisively illustrate the distinction between homotopy and spherical equivalence. 
Given two rational but spherically inequivalent maps,
a homotopy between them must contain {\it uncountably many} spherically inequivalent maps.
It follows that the four maps of Faran from ${\bf B}_2$ to  ${\bf B}_3$ are not homotopic through rational maps
in target dimension three, although they are homotopic in target dimension 
five.
This particular result provides a new method for establishing that two rational proper maps are homotopically
inequivalent through rational maps.

To further illuminate the situation for rational homotopies, in section 4 
we connect our discussion to the
so-called $X$-variety. The method for computing this variety from 
[D1] enables us to compute
it simultaneously for all the maps in a rational homotopy.

The first author acknowledges support from NSF Grants DMS 1066177 and DMS 
1361001. The second author
acknowledges support from NSF Grant DMS 
1362337.  Both authors acknowledge AIM.
We put the finishing touches on this paper at an AIM workshop in 2014, 
and discussed related ideas at an earlier AIM workshop.

\section{definitions and basic properties of homotopy equivalence}

We will denote the squared  Euclidean norm on ${\bf C}^n$ by $|| \ ||^2$ without
indicating the dimension. A holomorphic map $f:{\bf B}_n \to {\bf B}_N$ is 
{\it proper} if and only if
$||f(z)||^2$ tends to $1$ when $||z||^2$ tends to $1$.
It follows by standard complex analysis that $N\ge n$.
It is also easy to see that the composition of proper mappings between
balls is itself proper.
Let $U(n)$ denote the group of unitary transformation of ${\bf C}^n$.
Such transformations are of course the simplest examples of 
automorphisms of ${\bf B}_n$.
We note that $U(n)$, as a connected Lie group, is path connected.

Proper maps $f$ and $g$ between balls
are {\it spherically equivalent} if there are automorphisms $\phi$ (of
the domain ball) and  $\chi$ (of the target ball)
such that $g = \chi \circ f \circ \phi$.
Proper maps $f$ and $g$ are {\it norm equivalent} if $||f||^2 =
||g||^2$ as functions; this concept
provides an equivalence relation for maps with possibly different
target dimensions. When the target dimensions
are the same, norm equivalence is a special case of spherical
equivalence in which the domain automorphism is the identity
and the target automorphism  is unitary.

Let $f$ be a holomorphic mapping with values in ${\bf C}^N$. Its {\it 
embedding
dimension} is 
the number of linearly independent components of $f$. An 
equivalent
definition is the rank of the function $||f||^2$; in other words, the
smallest possible number
of terms in this squared norm.

At least
three versions of {\it homotopy equivalence} between proper maps are
sensible. In some situations we assume that $f$ and $g$
have the same target dimension, whereas in others we allow
the target dimensions to differ. We also
sometimes wish to demand that the maps in the family have
an additional property, such as rationality.

By Proposition 2.4,
given proper maps $f$ and $g$ with the same domain ball, there is a
minimal $M$ for which
$f$ and $g$ are homotopy equivalent in target dimension $M$.

Consider a continuous function $H:{\bf B}_n \times [0,1]  \to {\bf C}^N$ which is assumed
to be holomorphic in the first variable. We write $H_t$ for the map $z \to H(z,t)$.
Since $H_t:{\bf B}_n \to {\bf C}^N$  is a holomorphic mapping, it can
be written as a power series
$$ H_t(z) = \sum {\bf c_\alpha}(t) z^\alpha \eqno (1) $$
in ${\bf B}_n$. The series, which we have expressed in multi-index notation,
converges uniformly on compact subsets of ${\bf B}_n$. 
It follows (see Proposition 3.2)
that each of these
coefficients ${\bf c_\alpha}$ depends continuously on $t$, and
we  say that $H_t$ is a {\it continuous family}. 
In our homotopy considerations,  
we assume, for each $t$, that $H_t$ is a proper mapping between 
balls. 

We introduce the following definitions of the various homotopy equivalences.

\begin{definition} Let $f,g: {\bf B}_n \to {\bf B}_N$ be proper
holomorphic mappings. Then $f$ and $g$ are {\it homotopic}
if, for each $t \in [0,1]$ there is a proper holomorphic mapping
$H_t:{\bf B}_n \to {\bf B}_N$ such that:
\begin{itemize}
\item $H_0 = f$ and $H_1 = g$.
\item $H_t$ is a continuous family.

\end{itemize}\end{definition}

\begin{definition} Let $f,g: {\bf B}_n \to {\bf B}_N$ be proper
holomorphic mappings. Then $f$ and $g$ are {\it homotopic through
rational maps}
if, for each $t \in [0,1]$ there is a proper holomorphic mapping
$H_t:{\bf B}_n \to {\bf B}_N$ such that:
\begin{itemize}
\item $H_0 = f$ and $H_1 = g$.
\item $H_t$ is a continuous family.
\item Each $H_t$ is a rational mapping.

\end{itemize}\end{definition}

When $n\ge 2$ in Definition 2.2, we can replace the condition of
rationality by demanding that each $H_t$ be smooth
on the closed ball. See Section 4.

The third definition is required for a full understanding. It is
sometimes important to identify a proper map $f$ between balls
with the map $f \oplus 0 = (f,0)$. We write $f \sim (f \oplus 0)$.
These maps are norm-equivalent, but they are not spherically equivalent
because the target dimensions differ.
In this way we take advantage of
the natural injection of a target ball into a ball in higher
dimensions.

\begin{definition} Let $f: {\bf B}_n \to {\bf B}_{N_1}$ and $g: {\bf
B}_n \to {\bf B}_{N_2}$ be proper holomorphic mappings.
Then $f$ and $g$ are {\it homotopic in target dimension $k$}
if, for each $t \in [0,1]$ there is a proper holomorphic mapping
$H_t:{\bf B}_n \to {\bf B}_k$ such that:
\begin{itemize}
\item $H_0 \sim f \oplus 0$ and $H_1 \sim g \oplus 0$.
\item $H_t$ is a continuous family.
\end{itemize}
\end{definition}

It is evident that each of the notions of homotopy equivalence is an
equivalence relation. We briefly motivate Definition 2.3.
Given the target dimensions $N$ and $K$,
and a larger integer $M$,
we have natural injections $j_1: {\bf C}^N \to {\bf C}^M$ and $j_2
:{\bf C}^K \to {\bf C}^M$ each given by
$j(\zeta) = (\zeta, 0)$. The definition asks that the maps $j_1 \circ
f$ and $j_2 \circ g$ be homotopy equivalent.
Consider the simple example $H_t:{\bf B}_1 \to {\bf B}_2$ given by
$$ H_t(z) = (tz, \sqrt{1-t^2} z^2). $$
Then $H_1(z) = (z,0)$ and $H_0(z) = (0,z^2)$. Note that $(z^2,0) = U(0,z^2)$
for some unitary $U$. We would like to say
that $z$ and $z^2$ are homotopic in target dimension $2$,
and hence we naturally identify $z$ with $(z,0)$ and $z^2$ with $(z^2,0)$.

The following
decisive result in one dimension holds:

\begin{proposition} Suppose $f:{\bf B}_1 \to {\bf B}_1$ is proper holomorphic.
Then there is a unique positive integer $m$ such that
$f$ is homotopic in dimension $1$ to the map $z \mapsto z^m$. \end{proposition}

\begin{proof} It is well-known that each proper holomorphic self-map
of the unit disk is a finite Blaschke product:
$$ f(z) = e^{i \theta} \prod_{j=1}^m {z - a_j \over 1 - {\overline
a_j} z}. \eqno (2) $$
For each $j$, the point  $a_j$ in (2) satisfies $|a_j| < 1$.
These points need not be distinct. For $t \in [0,1]$, define $H_t$ by
replacing $\theta$ in (2) with $(1-t)\theta$
and each $a_j$ in (2)  with $(1-t) a_j$.
Then each $H_t$ is proper,  $H_0 = f$ and $H_1 = z^m$. The continuity in $t$ is evident.
Hence $f$ is homotopic to $z^m$, where $m$ is the {\it degree} of the
Blaschke product.
Next we note the uniqueness. The maps $z^m$ and $z^d$ cannot be
homotopic if $m \ne d$ because
we can recover the exponent $m$ by a line integral:
$$ m = {1 \over 2 \pi i} \int_{|z|=1} {f'(z) \over f(z)} dz. $$
As usual, an integer-valued continuous function is locally constant. \end{proof}

The number $m$ is the degree of the rational function $f$;
it is also the degree of the divisor defining the zero-set of $f$.

Perhaps the most surprising result of this paper is that the degree
of a proper rational mapping between balls is {\bf not invariant} under 
homotopy. The following example illustrates this point and suggests
ideas from the last section of the paper.

\begin{example} We define proper polynomial maps $f,g$ from ${\bf B}_2$
to ${\bf B}_5$. Both these maps have embedding dimension $5$.
These maps are of different degree but they are homotopic in
target dimension $5$.

$$ f(z,w) = (z,zw,zw^2,zw^3,w^4). $$
$$ g(z,w) = (-w^2,zw,-zw^2,z^2w,z^2). $$
Since each of $f$ and $g$ is a monomial map with five distinct monomials,
the embedding dimension in each case is $5$. We check that they are 
endpoints of a one-parameter family of proper maps.

First define a proper map $h: {\bf B}_2 \to {\bf B}_3$ 
by $h(z,w) = (z,zw,w^2)$. Next define a unitary matrix $U$ on ${\bf C}^3$ 
by
$$ U= \begin{pmatrix} {\rm cos}(\theta) & 0 & -{\rm sin}(\theta) \cr 0
& 1 & 0 \cr {\rm sin}(\theta) & 0 & {\rm cos}(\theta) \end{pmatrix}.
$$ 
Finally let $W: {\bf B}_3 \to {\bf B}_5$ be the Whitney map defined by
$$ \zeta \to (\zeta_1,\zeta_2,\zeta_1 \zeta_3, \zeta_2 \zeta_3, 
\zeta_3^2). $$
Put $t={\rm cos}(\theta)$. Define $H_t$ by $H_t = W \circ U \circ h$;
then $H_t:{\bf B}_2 \to {\bf B}_5$. Since the composition of proper
maps is proper, each $H_t:{\bf B}_2 \to {\bf B}_5$ is proper. Writing
$c$ for ${\rm cos}(\theta)$
and $s$ for ${\rm sin}(\theta)$, we obtain
$$ H_t(z,w) = (cz - sw^2, zw, (cz-sw^2) (sz+cw^2), zw(sz+cw^2),
(sz+cw^2)^2). \eqno (3) $$
When $t=0$ in (3) we obtain $f$ and when $t=1$ in (3) we obtain $g$.

\end{example}

Other natural numerical invariants such as the maximum number of
inverse images of a point also fail to be invariant under homotopy.

\begin{remark} In Example 2.1, the maximum
number of inverse images of a point is {\it not} a homotopy invariant,
even though all maps in the family
have the same embedding dimension. Note that $f^{-1} (0,0,0,0,a)$ has
four points for $a \ne 0$.
No point in the image of $g$ has more than two inverse images. 
 \end{remark}

The next example does behave as in the one-dimensional case.

\begin{example} Each  automorphism $\phi$ of ${\bf B}_n$
 is homotopic to the identity map.
The proof is easy; each automorphism is a composition of a unitary
transformation $U$ and a linear fractional
automorphism of the form
$$z \to {L_a(z) - a \over 1 - \langle z,a\rangle}. $$
Here $L_a$ is a linear map depending continuously on $a$, and $a$
is a point in the unit ball.
By multiplying $a$ by $1-t$, and deforming $U$ into the identity, we
obtain a family $H_t$
where $H_0(z) = \phi(z)$ and $H_1 (z)=z$.
\end{example}

The following result relates the various equivalence relations.

\begin{proposition} Let $f,g:{\bf B}_n \to {\bf B}_N$
be proper holomorphic maps.
\begin{itemize}
\item If $f$ and $g$ are norm equivalent, then they are spherically
equivalent. The converse fails.
\item If $f$ and $g$ are spherically equivalent, then they are
homotopy equivalent. The converse fails.
\end{itemize}\end{proposition}
\begin{proof} By [D], the equality  $||f||^2=||g||^2$ implies that
there is a unitary map $U$
such that $g=Uf$ and hence $f$ and $g$ are spherically equivalent. The
converse fails: if $f$ is the identity
and $g$ is an automorphism which moves the origin, then $f$ and $g$
are spherically equivalent but not norm
equivalent. Consider the second statement. Suppose there are automorphisms
$\chi$ and $\phi$ for which $g = \chi \circ f \circ \phi$.
We may then deform each automorphism as in Example 2.2 to obtain a
homotopy between $f$ and $g$.
The converse fails: consider a Blaschke product $f$ with three
distinct factors. By Proposition 2.1, $f$
is homotopic to $z^3$; it is easy to see that $f$ is not spherically
equivalent to $z^3$.
The same idea works in higher dimensions upon replacing product by
tensor product. \end{proof}

Homotopy in the equi-dimensional case is easy. Proposition 2.1 handled
the one-dimensional case.
For $n\ge 2$ we have:

\begin{proposition} For $n \ge 2$, let $f:{\bf B}_n \to {\bf B}_n$
be a proper holomorphic map.  Then $f$ is homotopic to the identity.
\end{proposition}
\begin{proof} By a well-known result of Pincuk, $f$ must be an
automorphism. By Example 2.2,
$f$ is homotopic to the identity. \end{proof}

We note that if $f$ and $g$ are homotopy equivalent
in target dimension $M_0$, then they are homotopy equivalent in 
target dimension $M$ if $M \ge M_0$.

We have the following simple result, noted years ago (in different
language) by the first author in [D]. The map $H_t$ in the proof of this 
proposition is called the {\it juxtaposition} of $f$ and $g$.

\begin{proposition} Let $f:{\bf B}_n \to {\bf B}_N$ and $g:{\bf B}_n
\to {\bf B}_K$
be proper holomorphic maps. Then $f$ and $g$ are {\it homotopic in
target dimension $M$} if $M \ge n + {\rm max}(N,K)$.
Furthermore, if $f$ and $g$ are rational, then the same conclusion holds
with homotopy replaced by homotopy through rational maps.
\end{proposition}
\begin{proof} First define $H_t$ by $H_t = \sqrt{1-t^2}f \oplus t g$. Then
$$||H_t(z)||^2 = (1-t^2)||f(z)||^2 + t^2 ||g(z)||^2. \eqno (4) $$
When $||z||$ tends to one, $||H_t(z)||$ also does.
It follows that $H_t$ is a proper mapping from ${\bf B}_n$ to ${\bf B}_{N+K}$.
The continuity in $t$ is obvious. Formula (4) makes evident the norm
equivalence at the endpoints $0$ and $1$.
Thus $f$ and $g$ are homotopy equivalent in target dimension $N+K$. We
can lower this dimension to $n + {\rm max}(N,K)$.
To do so, let $I$ denote the identity mapping. By the same reasoning,
$f$ is homotopic to $I\oplus 0$ in dimension $n+N$
and $g$ is homotopic to $I\oplus 0$ in dimension $n+K$. Since homotopy is 
an equivalence relation,
$f$ and $g$ are homotopic in target dimension $n + {\rm max}(N,K)$,
and hence also for any larger target dimension. 
When $f$ and $g$ are rational, the same argument 
provides a rational homotopy between them. \end{proof}

\section{homotopy equivalence through rational maps (general results)}

In Example 3.1 below, the number of spherical equivalence classes is finite.
In general this number is infinite. For example, as soon as $N \ge 2n$,
there are one-dimensional families
of spherically inequivalent proper mappings. These mappings can be
taken to be quadratic polynomials. See pages 168-169
of [D] and also [L]. Quadratic polynomial proper maps also appear in [JZ].
By contrast, when the domain dimension is at least
$2$, we will show in Theorem 3.1 that 
the number of {\bf homotopy classes}
for rational proper mappings, in each target dimension,  is finite.

\begin{proposition}
Fix integers $d$ and $n$. There exists a constant $B=B(n,d)$ such that
the following holds: If $h:{\bf B}_n \to {\bf B}_N$ is a {\it proper} rational map
of degree at most $d$, written as $h = \frac{p}{q}$ with $p,q$ of degree at most $d$,
where $q(0) = 1$, and $q$ is not zero on ${\bf B}_n$,
then the coefficients of (the polynomials) $p$ and $q$
are bounded by $B$.
\end{proposition}

\begin{proof}
Let $p$ and $q$ be as in the statement.  Let us first show that
coefficients of $q$ are bounded.  It is enough to assume that $q(0)=1$ and
$q$ is nonzero on the unit ball.  In one dimension, the claim follows
by simply factorizing $q$ as
$$
q(z) = \prod_j \left(1-z {\overline a_j}\right).
$$
Then the $a_j$ are all of modulus at most $1$ and 
the coefficients
of $q$ are in fact bounded by the largest binomial coefficient of the form ${d \choose k}$.

Next, in $n$ dimensions, decompose $q$ into homogeneous parts as $q(z) = 1 +
q_1(z) + q_2(z) + \cdots + q_d(z)$.  For each $z \in {\bf B}_n$ we have 
$$
q(\xi z) = 1 + \xi q_1(z) + \xi^2 q_2(z) + \cdots + \xi^d q_d(z).
$$
This polynomial in $\xi \in {\bf C}$ has no zeros on the unit disk.
Therefore all its coefficients $q_j(z)$ are bounded by $B(1,d)$ by the argument above.  Hence
$|q(z)| \leq (d+1) B(1,d)$ for $z \in {\bf B}_n$.  Via the Cauchy formula all the
coefficients of $q$ are bounded by some constant depending only on $n$ and
$d$.

Next, since $||h|| = 1$ on the unit sphere, we have
$$
||p(z)|| = |q(z)| \leq d B(n,d)
$$
for $z$  in
the unit sphere.  Again via the Cauchy formula, the coefficients of $p$
are bounded by a constant only depending on $n$ and $d$.
\end{proof}

The proposition says that we can always find a representative $\frac{p}{q}$ with
bounded coefficients, while at the same time normalizing with $q(0) = 1$ and
not requiring $p$ and $q$ to be
in lowest terms.  The following set will be useful.
\begin{definition} Fix a positive integer $d$. Let ${\mathcal R}$
 denote the set of polynomial maps
$$ (p,q):{\bf C}^n \to {\bf C}^N \times {\bf C}$$
with the following properties:
\begin{itemize}

\item The degree of $q$ and of each component polynomial of $p$ is at most $d$.

\item $q(0)=1$, and $q$ is not zero on ${\bf B}_{n}$.

\item ${p \over q}: {\bf B}_n \to {\bf B}_N$  is proper.
\end{itemize}
\end{definition}

Proposition 3.1 implies that the set
${\mathcal R}$ is relatively compact.
Let
${\mathcal R}^*$ denote the closure of
${\mathcal R}$.  The first property is obviously preserved
under closure. The second property is preserved because of the Hurwitz theorem;
since $q(0)=1$, the limit function is not identically $0$, and hence never $0$.
The condition that $\frac{p}{q}$
must be a proper map need not be preserved, but we can easily identify the limits.
By continuity, if $(p,q) \in {\mathcal R}^*$,
we must have $||p||^2 = |q|^2$ on the unit sphere. 
Then $\frac{p}{q}$ is a proper map of balls if it is nonconstant.  Hence,
${\mathcal R}^*$ is a compact set that includes ${\mathcal R}$
and also the polynomials $(p,q)$, where $\frac{p}{q}$ is constant.

Given a sequence $\{ \omega_j \}$ of positive weights, we let $\ell^1(\omega)$
denote the space of sequences  $\{ x_j \}$ for which
$$\sum \omega_j |x_j| < \infty.$$
We say that $\{x_j\}$ is in weighted $\ell^1$ with weights $\omega_j$.
We will use this space when the sequence is indexed by multi-indices.

Let $H_t:{\bf B}_n \to {\bf B}_N $ be a family of proper mappings between balls.
We can expand each $H_t$ as a power series converging uniformly on compact
subsets
of ${\bf B}_n$.  If we assume that the function $t \to H_t(z)$ is continuous,
it follows that each Taylor coefficient depends continuously on $t$.
It also follows
that the Hermitian matrix of Taylor coefficients of $||H_t||^2$
depends continuously on $t$.

Here we do not require the homotopies to
have any regularity on the boundary. Let $H_t$ be a family of
holomorphic mappings defined on
the unit ball. We write
$$  H_t(z) = \sum {\bf c}_\alpha(t) z^\alpha \eqno (5) $$
$$ {|| H_t(z) ||}^2 = \sum_{\alpha,\beta} a_{\alpha\beta}(t) z^\alpha
\bar{z}^\beta. \eqno (6) $$
Thus $a_{\alpha \beta}(t) = \langle c_\alpha(t), c_\beta(t) \rangle$.

\begin{proposition}
Let $H_t(z) \colon {\bf B}_n \to {\bf B}_N$ be a homotopy of
proper maps.  The Taylor coefficients $a_{\alpha \beta}(t)$ of
$||H_t||^2$ and ${\bf c}_\alpha(t)$ of $H_t$ are continuous functions
of $t$.
In fact, there exist fixed weights $\{ \omega_{\alpha\beta} \}$ such that
the map $t \mapsto [ a_{\alpha\beta}(t) ]_{\alpha\beta}$
is a continuous map from $[0,1]$ to $\ell^1(\omega)$.
\end{proposition}

\begin{proof}
For each $\alpha$, the coefficient  ${\bf c}_\alpha(t)$ is given by a
Cauchy integral
over the distinguished boundary $T$ of a polydisc $P$ centered at $0$
whose closure lies in ${\bf B}_n$:

$$ {\bf c}_\alpha(t) = {1  \over (2 \pi i)^n} \int_T {H_t (\zeta) \over
\zeta^{\alpha+1}} d \zeta. \eqno (9) $$
Since $H_t$ is continuous in $t$, also each $c_\alpha$ is continuous
in $t$. Furthermore, for a constant $C_\alpha$,
(9) yields the estimates
$$ ||{\bf c}_\alpha(t) - {\bf c}_\alpha(s)|| \le C_\alpha \ {\rm sup}_{z
\in T} ||H_t(z) - H_s(z)|| \eqno (10) $$
$$ ||{\bf c}_\alpha(t)|| \le C_\alpha. \eqno (11)$$
Note that (11) holds for all $t$.
Next consider the coefficients $a_{\alpha \beta}$ in the squared norm.
There are  constants $C_{\alpha \beta}$ such that
$$
{|| a_{\alpha\beta}(t) ||} = | \langle c_\alpha , \overline{c_\beta}
\rangle | \leq C_\alpha C_\beta.  $$
For an appropriate choice of weights $\omega_{\alpha \beta}$
(independent of $H$)
we obtain that
$[ a_{\alpha\beta} ]$ is in $\ell^1(\omega)$.

Estimate
$$
|a_{\alpha\beta}(t)-a_{\alpha\beta}(s)| \leq
|| {\bf c}_\alpha (t) - {\bf c}_\alpha(s)|| \  || {\bf c}_\beta (t)||
+ ||  {\bf c}_\beta (t) - {\bf c}_\beta(s) ||  \  || {\bf c}_\alpha (s) || $$
$$   \leq C_{\alpha \beta}
\sup_{z \in T} || H_t(z) - H_s(z) || .
$$

As $s$ tends to $t$,
$\sup_{z \in T} || H_t(z) - H_s(z) ||$ tends to zero.  Therefore,
for the same weights, the map
$t \mapsto [ a_{\alpha\beta} ]$ is continuous from $[0,1]$ to $\ell^1(\omega)$.
\end{proof}

The following lemma is used in verifying the finiteness of the homotopy
classes.  Given a rational family $H_t$, for each $t$ we can write
$H_t = {p_t \over q_t}$
for a polynomial $q_t$ and a polynomial map $p_t$. Since this choice
is not unique, we need to establish
continuity in $t$.

\begin{lemma}
Let $H_t$ be a homotopy of rational functions of degree
at most $d$ in target dimension $N$.  Then
there exists a choice of degree $d$
polynomial maps $p_t \colon {\bf C}^n \to {\bf C}^N$ and $q_t \colon {\bf
C}^n \to {\bf C}$
whose coefficients are continuous functions of $t$ such that
$$
z \mapsto \frac{p_t(z)}{q_t(z)} $$
is a rational proper mapping of ${\bf B}_n$ to ${\bf B}_N$. Furthermore
$$
H_0 = \frac{p_0}{q_0}, \qquad \text{and} \qquad H_1 = \frac{p_1}{q_1}. $$
\end{lemma}

Note that the homotopy obtained in the proof is not necessarily the same as $H_t$.

\begin{proof}
Let ${\mathcal R}$ be as in Definition 3.1 and let
${\mathcal R}^*$ be its closure as above.

Let $\Phi$ be the map from ${\mathcal R}^*$ to the weighted $\ell^1$ defined 
by letting $\Phi(p,q)$ be the sequence of Taylor coefficients at $0$
of $|| \frac{p}{q} ||^2$.
The map $\Phi$ is continuous, 
by similar estimates as in the proof of Proposition 3.2.
If $\frac{p}{q}$ is a proper map  of balls, then
by [CS] $\frac{p}{q}$ extends
holomorphically past the sphere.
Thus $\frac{p}{q}$ has only removable
singularities and stays bounded on the sphere and hence on the torus $T$ (as
defined above).

Let $(p_1,q_1)$ and
$(p_2,q_2)$ be two elements in ${\mathcal R}^*$.  These two elements
are close together in the standard topology on ${\mathcal R}^*$
if their coefficients are close together.  Write $q_j(z) = 1-Q_j(z)$,
where $Q_j(0) = 0$.  Suppose $Q_1$ is close enough to $Q_2$, such that
on $T$
$$
|q_1(z)q_2(z)| = |1-Q_1(z)-Q_2(z)+Q_1(z)Q_2(z)| > \frac{1}{2} .
$$
Then suppose that $(p_1,q_1)$ and $(p_2,q_2)$ are close enough so that
on $T$
$$
\| p_1(z)q_2(z)-p_2(z)q_1(z) \| < \epsilon .
$$
Then
$$
\left\| \frac{p_1(z)}{q_1(z)} - \frac{p_2(z)}{q_2(z)} \right\| < 2\epsilon .
$$
In other words,
the sup norm
$$
\sup_{z\in T} \left\| \frac{p_1(z)}{q_1(z)}- \frac{p_2(z)}{q_2(z)} \right\|
$$
can be bounded in terms of the difference of the coefficients of
$(p_1,q_1)$ and $(p_2,q_2)$.

Then the Taylor coefficient
$$
\frac{1}{(2\pi i)^n}
\int_{T} \frac{p(\zeta)/q(\zeta)}{\zeta^{\alpha+1}} d\zeta
$$
can be bounded in terms of the sup norm of $\frac{p}{q}$
over $T$.
Hence, the map $\Phi$ is continuous.

Fix $(P,Q) \in {\mathcal R}^*$.
Next we consider the fiber $\Phi^{-1}\bigl(\Phi(P,Q)\bigr)$.  That is, the
set of all choices $(p,q) \in {\mathcal R}^*$, such that
$|| \frac{p}{q} || = || \frac{P}{Q} ||$.

It is easy to check that the set of all $(p,q) \in {\mathcal R}^*$ with
$\frac{p}{q}=\frac{P}{Q}$ is convex.
If $||h_1|| = ||h_2||$ for two maps of spheres $h_1$ and $h_2$,
then there is a unitary matrix $U$ such that $h_1 = U h_2$ (see [D]),
and the set of unitary matrices is connected.  Therefore,
the fiber $\Phi^{-1}\bigl(\Phi(P,Q)\bigr)$ is connected.

Our homotopy $H_t$ is represented by a
path $\gamma \colon [0,1] \to \ell^1(\omega)$ given by the coefficients
of $||H_t(z)||^2$.  The set $\Phi^{-1}(\gamma)$ is closed in ${\mathcal R}^*$
and therefore compact. It is also connected: if it were disconnected it
would be a union of two disjoint relatively open sets $X_1$ and $X_2$.
For every $t$, as the fiber is connected, $\Phi^{-1}\bigl(\gamma(t)\bigr)$
is a subset of $X_1$ or $X_2$ but not both.  We could therefore
obtain $[0,1]$ as a union of two disjoint compact sets $\Phi(X_1) \cup
\Phi(X_2)$, a contradiction.

The set $\Phi^{-1}(\gamma)$ is therefore
connected.  Furthermore
$\Phi^{-1}(\gamma) \subset {\mathcal R}$.
We next claim
that every connected topological component of $\mathcal R$
is path connected.  This fact follows by showing that
$\mathcal R$ is a semialgebraic set (a set defined by polynomial equations
and inequalities).
The set ${\mathcal R}^*$ is
an intersection of
a closed real algebraic subvariety of the space of polynomials $(p,q)$
of degree at
most $d$ with a polydisc of fixed radius $B$, and thus ${\mathcal R}^*$
is semialgebraic.  The set of polynomials
$(p,q)$ such that $\frac{p}{q}$ is a constant (the set ${\mathcal R}^*
\setminus {\mathcal R}$) is also a subvariety of the
space of polynomials, and hence ${\mathcal R}$ is semialgebraic.

Then there must exist a path in ${\mathcal R}$
from $(p_0,q_0)$ to $(p_1,q_1)$ where
$H_0 = \frac{p_0}{q_0}$ and $H_1 = \frac{p_1}{q_1}$.
This path yields the desired $p_t$ and $q_t$.
\end{proof}

\begin{theorem} Let $S$ denote the set of homotopy classes (of
rational mappings and in target dimension $N$) of proper
rational mappings $f:{\bf B}_n \to {\bf B}_N$.  Assume that $n\ge 2$.
Then $S$ is a finite set. (For $n=1$, $S$ is countable by Proposition
2.1.) \end{theorem}

\begin{proof} We first need to know that a degree bound holds. That
is, for $n \ge 2$ and $f$ as in the statement of the theorem,
the degree of $f$ is bounded by some expression $c(n,N)$. The sharp bound
is not known, [LP], but any bound will do here. For example, by [DL1]
$$ d \le {N(N-1) \over 2(2n-3)}.  $$

Let $f = {p \over q}:{\bf B}_n \to {\bf B}_N$ be a rational proper
map, reduced to lowest terms, and of degree $d$.
We may write
$$ p(z) = \sum_{|\alpha|=0}^d {\bf C}_\alpha z^\alpha, \eqno (12.1)$$
$$  q(z) = \sum_{|\alpha|=0}^d D_\alpha z^\alpha,\eqno (12.2)$$
where ${\bf C}_\alpha \in {\bf C}^N$ and $D_\alpha \in {\bf C}$.
Consider the Hermitian polynomial $R$ defined by $R=||p||^2 - |q|^2$.
It is of degree at most $d$ in $z$ and of total degree at most $2d$, and it
is divisible by $||z||^2 -1$. A proper holomorphic mapping $f$ of
degree at most $d$
thus determines a Hermitian form
$$ R(z, {\overline z}) = \sum_{ |\alpha|, |\beta| \le d} c_{\alpha
\beta} z^\alpha {\overline z}^\beta \eqno (13) $$
on the vector space of polynomials of degree at most $d$. Note that
$$ c_{\alpha \beta} = \langle {\bf C}_\alpha, {\bf C}_\beta \rangle  - D_\alpha
{\overline D_\beta}. \eqno (14) $$

We next find the explicit condition for the expression in (13) to
vanish on the sphere.
Put $z_j = r_j e^{i \theta_j}$. Thus $r =(r_1,...,r_n) = (|z_1|,...,|z_n|)$ and
$\theta = (\theta_1,...,\theta_n)$. Assume that (13) vanishes on the
sphere. Equating Fourier coefficients
shows the following: on the set $S$ defined by $\sum r_j^2 = 1$, and for
each multi-index $\nu$, we have
$$  \sum_\beta (\langle {\bf C}_{\beta + \nu}, {\bf C}_\beta \rangle -
D_{\beta +
\nu}{\overline D_\beta}) r^{2\beta}= 0. \eqno (15) $$
By putting $x_j=r_j^2$ we can regard these conditions as equalities on
the hyperplane $\sum x_j =1$.

Thus a proper holomorphic rational mapping $f$ of degree at most $d$ determines
a Hermitian form
$$ \sum_{ |\alpha|, |\beta| \le d} c_{\alpha \beta} z^\alpha
{\overline z}^\beta. $$
The space of forms is a finite-dimensional real vector space $V$.
The conditions in (10) are linear in the coefficients $c_{\alpha \beta}$,
and hence determine a subspace of $V$. The forms are restricted
further because each such
form must have at most $N$ positive and exactly one negative
eigenvalue. This restriction
is determined by finitely many polynomial inequalities on the
$c_{\alpha \beta}$.
Since dividing all the coefficients by
the same constant does not change the proper map $f$, we may assume
that the squared norm
of the coefficients equals one. Hence (the norm equivalence class of)
a proper rational map corresponds to the intersection of the unit sphere
in a finite-dimensional real vector space with a set described by
finitely many polynomial inequalities.
Such a set can have at most a finite number of components. By the lemma, each
component corresponds to a collection
of homotopic rational proper mappings with target dimension at most $N$.
\end{proof}

\begin{theorem} Assume $n\ge 2$.
Let  $H_t \colon {\bf B}_n \to {\bf B}_N$ be a homotopy of
rational proper maps.
Fix $t_0 \in [0,1]$. The set
$$ \{ t \in [0,1] : H_t \text{ is spherically equivalent to } H_{t_0} \}$$
is closed in $[0,1]$.
\end{theorem}

\begin{proof}
Let $H_{t_0} = {p_0 \over q_0}$. We must show that the set of $t$ for which $H_t$ is spherically equivalent
to $H_{t_0}$ is closed. To do so, we must determine all rational maps spherically equivalent to a given one.

The degree bounds imply that the degrees of $H_t$ for $t \in [0,1]$
are uniformly bounded by some $d=d(n,N)$.  Let ${\mathcal R}$ be as in Definition 3.1.
As before, let ${\mathcal R}^*$ denote the closure of ${\mathcal R}$ in
the space of all polynomial mappings.  The set ${\mathcal R}^*$ is compact.  Define $\Phi$
as in Lemma 3.1. Thus $\Phi: {\mathcal R}^* \to \ell^1(\omega)$, and we have
shown that $\Phi$ is continuous.  Since ${\mathcal R}^*$
is compact, $\Phi$ is a closed map.

Next we want to see the effect of composition on both sides by automorphisms.
Each domain automorphism has the form
$$ U {a - L_a (z) \over 1 - \langle z,a \rangle} $$
for $U \in {\mathcal U}(n)$ and $a \in {\bf B}_n$. Here
$$ L_a(z) = {\langle z,a \rangle a \over s+1} + s z  $$
and $ s^2 = 1 - ||a||^2$. Thus $Aut({\bf B}_n)$ can be identified with
${\mathcal U}(n) \times {\bf B}_n$. We compactify by allowing $a$ to lie in the unit sphere.
When $||a||=1$, we see that $s=0$ and 
$$ {a - L_a(z) \over 1 - \langle z,a \rangle} = {a - \langle z, a \rangle a \over 1 - \langle z,a \rangle} = a. $$
Therefore the only new maps in ${\overline {Aut({\bf B}_n)}}$ after compactification are constants. We may identify an automorphism with
the linear map on ${\bf C}^{n+1}$ given by
$$ (z,w) \mapsto ( wa - L_a z, w - \langle z,a \rangle),$$
and the constant map with the linear map
$$ (z,w) \mapsto ( a(w- \langle z,a\rangle), w - \langle z,a \rangle). $$

We do the same construction in the target ${\bf C}^N$. Composition on both sides with automorphisms (and with the
degenerate maps in the closure) now
becomes a continuous map
$$ {\bf T} : {\mathcal R}^* \times {\overline {Aut({\bf B}_n)}} \times {\overline {Aut({\bf B}_N)}} \to {\mathcal R}^*. $$

This map ${\bf T}$ is continuous from a compact set to a compact set, and therefore maps
closed sets to closed sets. Since the maps with $||a||=1$ correspond to constant maps,
the image of the automorphisms lies in ${\mathcal R}$.

As before, the  homotopy $H_t$ defines a curve $\gamma:[0,1] \to \ell^1(\omega)$ given by the coefficients
of $||H_t(z)||^2$. Let $X$ denote the image 
$$ {\bf T} \big( (p_0,q_0 ) \times {\overline {Aut({\bf B}_n)}} \times {\overline {Aut({\bf B}_N)}} \big).$$ 
Then $X$ is closed, and $X \cap {\mathcal R}$
is the image under ${\bf T}$ of all projectivized maps spherically equivalent to $H_{t_0}$. 
The set $\Phi^{-1}(\gamma) \cap X$ is compact because, as in the proof of Lemma 3.1,  $\Phi^{-1}(\gamma) \subset {\mathcal R}$.
But $\Phi^{-1}(\gamma) \cap X$ is the set of projectivized maps spherically equivalent to $H_{t_0}$. Since $\Phi$ is continuous
and $X$ is closed, the set of maps in $H_t$ that are spherically equivalent to $H_{t_0}$ is a closed subset of $[0,1]$.

\end{proof}

A homotopy between two spherically inequivalent maps must contain
uncountably many spherically inequivalent maps.

\begin{corollary}
Suppose $H_t$ is a homotopy of proper rational maps between balls in
dimension $N$.
If $H_0$ and $H_1$ are not spherically equivalent, then $H_t$ contains
uncountably many spherically inequivalent maps.

In particular, the juxtaposition $J_t(f,g)$ of any two spherically inequivalent rational
maps always contains uncountably many inequivalent maps.
\end{corollary}

\begin{proof}
Each spherical equivalence class intersects the path in a
closed set, as we have shown above.  The
interval $[0,1]$ cannot be written as a union of countably many disjoint
closed sets, by Sierpinski's Theorem.  Hence, if there are at least two
inequivalent maps in the homotopy, then there must be uncountably many.
\end{proof}

\begin{corollary}
All four Faran maps from Example 3.1 below
are homotopically inequivalent in target
dimension $3$ through rational maps.
\end{corollary}

\begin{proof}
By Faran's theorem in [Fa], there are only 4 spherical equivalence
classes of rational
maps from ${\bf B}_2$ to ${\bf B}_3$.
\end{proof}

\begin{example} Here are
representatives of the four spherical equivalence classes:

$$ f(z,w) = (z,w,0) \eqno (16.1)$$
$$ g(z,w) = (z^2, zw, w) \eqno (16.2) $$
$$ h(z,w) = (z^2, \sqrt{2} zw, w^2) \eqno (16.3) $$
$$ \phi(z,w) = (z^3, \sqrt{3} zw, w^3). \eqno (16.4) $$

 By Corollary 3.2, none of these maps are homotopy equivalent (through
rational maps) in target dimension
$3$. By Proposition 2.4, all are homotopy equivalent in dimension $5$.
We analyze what happens in dimensions $4$ and $5$.

The maps $f$ and $g$ are homotopic in dimension $4$:

$$ H_t(z,w) = (\sqrt{1-t^2}z, tz^2, t zw, w). $$

The maps $g$ and $h$ are homotopic in dimension $4$:

$$ H_t(z,w) = (z^2, \sqrt{2 - t^2} zw, tw, \sqrt{1-t^2} w^2). \eqno (17)
$$
It follows that $f,g,h$ are members of the same equivalence class in
dimension $4$.
We suspect that $\phi$ is not, a question we posed at AIM in June 2014.
Note that $h$ and $\phi$ are homotopy equivalent in dimension
$5$; here is an explicit homotopy:
$$ H_t(z,w) = (tz^2, tw^2, \sqrt{1-t^2} z^3, \sqrt{1-t^2}w^3,
\sqrt{3-t^2} zw).  $$
Thus, in target dimension $5$ these four maps
lie in the same homotopy class, whereas in target
dimension $3$ they lie in $4$ distinct homotopy classes.
\end{example}

\begin{remark} By Corollary 3.1, the family connecting $f$ and $g$ in Example 3.1 consists
of spherically inequivalent maps.  This result also follows from an old result in [D2]. 
If polynomial proper maps preserve the origin and are spherically equivalent, then they must
be unitarily equivalent. It is easy to check that unitary equivalence fails 
for the maps in Example 3.1. \end{remark}

\section{Rational proper mappings between balls}

Example 2.1 reveals that the degree of a rational proper mapping between 
balls
is {\bf not} a homotopy invariant. It is natural to further investigate
the rational case, because of the following theorem
of Forstneric [F1]. One of the key tools in the proof is a variety
called the $X$-variety of $f$.
The method of proof involves the Schwarz reflection principle. The
variety $X_f$ contains the graph of $f$,
and Forstneric showed that this variety is rational. Assuming the map
is rational, the first author developed in [D1]
an efficient method for computing this variety, which we will recall
and use in this section.
The method allows us to understand how this variety depends on $t$
when $H_t$ is a family of rational proper mappings.

\begin{theorem}[Forstneric 1989] Assume $n\ge 2$.
Suppose $f:{\bf B}_n \to {\bf B}_N$ is proper, holomorphic,
and smooth up to the boundary. Then $f$ is a rational mapping. \end{theorem}

By a follow-up result of Cima-Suffridge ([CS]),  $f$ extends
holomorphically past the sphere.
Hence, for $n\ge 2$, we can identify rational proper mappings between
balls with holomorphic proper
mappings between balls that extend smoothly to the unit sphere.

We now develop the properties of $X_f$.
Assume $R(z, {\overline z})$ and $r(z,{\overline z})$ are
real-analytic, and $\{r=0\}$ is a hypersurface.
Suppose $R$ vanishes on $\{r=0\}$.
Then $R(z,{\overline w}) = 0$ on the set defined by $r(z,{\overline w})=0$.
This result is known as {\it polarization}.

Let $f:{\bf B}_1 \to {\bf B}_1$ be proper (hence a finite Blaschke product).
By polarization, if we know the value of
$f(z)$ for some $z$ inside the circle, then we automatically know
the value of $f$ at the reflected point ${1 \over {\overline z}} = {z
\over |z|^2}$ from the formula

$$ f({z \over |z|^2}) = {1 \over {\overline {f(z)}}}.  $$

One of the difficulties in homotopy considerations is that
reflection in higher dimensions is much more subtle.
Suppose $n\ge 2$, and that $f:{\bf B}_n \to {\bf B}_N$ is proper and
smooth up to the sphere.
Then $f$ is rational and holomorphic past the sphere. What do we get
from polarization and reflection?

$$ \langle z,w \rangle = 1 \implies \langle f(z),f(w) \rangle = 1. $$
Given $z$ we know
$$ \langle f(z), f({z \over ||z||^2})\rangle = 1, $$
but this equation does not determine the value of $f$ at the reflected point.

\begin{definition} (Forstneric) Suppose $f:{\bf B}_n \to {\bf B}_N$ is proper.
$$ X_f = \{(w, \zeta) \in {\bf C}^n \times {\bf C}^N : \langle z,w
\rangle = 1 \implies \langle f(z),\zeta \rangle = 1.\} \eqno (X) $$
\end{definition}

It is convenient to decree that $(0,f(0)) \in X_f$.
In (X) we could insist
that $w$ lie in the domain of $f$, or we could allow $\infty$.
For us $X_f$ will be
the union of the set defined by (X) with the point $(0,f(0))$,
assuming that $f(w)$ is defined.

Note that $(w,f(w)) \in X_f$ by polarization. In general $X_f$
is a proper superset of the graph. If $N$ is minimal for $f$, then
the fiber over a generic $w$ will be $f(w)$, but in most cases
exceptional fibers
exist. If $N$ is larger than the embedding dimension of $f$, then the fibers
are all positive-dimensional.
Saying that $X_f$ equals the graph of $f$ amounts to saying
$f(w)$ is the unique solution to the polarized equation.

The method in [D1] uses homogenization techniques to create a matrix
${\overline C} (w)$ of holomorphic polynomial functions with the
following properties.
Suppose $f:{\bf B}_n \to {\bf B}_N$ is rational of degree $d$.
The matrix $C(w)$ has $N$ columns.
It has $K(n,d)$ rows, where $K(n,d)$
is the number of homogeneous monomials of degree $d$ in $n$ variables.
Thus we can think of ${\overline C}(w)$ as a linear map from ${\bf
C}^N$  to $H^0({\bf P}_n, \mathcal O(d))$.

\begin{theorem} Let $f={p \over q} :{\bf B}_n \to {\bf B}_N$ be a
proper rational
holomorphic mapping.
Let $C({\overline w})$ denote the linear map from above.
For each nonzero $w$ in the domain of $f$, we have $(w,\zeta) \in X_f$
if and only if $\zeta - f(w)$ lies in the null space
of ${\overline C} (w)$.
Thus $X_f$ equals the graph of $f$ if and only if, for each nonzero
$w$ in the domain of $f$,
the null space of ${\overline C} (w)$ is trivial. 
Furthermore, if $H_t$ is a homotopy of rational mappings,
then the corresponding linear maps $C_t({\overline w})$ depends continuously on $t$. \end{theorem}

\begin{corollary} Let $f$ be a rational proper holomorphic mapping
between balls.
For each nonzero $w$ in the domain of $f$, the fiber $X_f(w)$ over $w$ is the
affine space $f(w) + {\rm null \ space} ({\overline C}(w))$.
In particular, the null space of $C(w)$ is trivial if and only if the
fiber over $w$
is zero-dimensional, when it is the single point $f(w)$. \end{corollary}

\begin{corollary} Let $H_t$ be a homotopy of rational proper maps. If the $X$-variety
of $H_{t_0}$ is the graph of $H_{t_0}$, then the same holds for $t$ near $t_0$.
\end{corollary}

\begin{example} Let $f:{\bf B}_2 \to {\bf B}_4$ be the group-invariant map
$$ f(z_1,z_2) = (z_1^5, \sqrt{5} z_1^3 z_2, \sqrt{5} z_1 z_2^2, z_2^5). $$
We compute $X_f$ as follows. Homogenize:

$$ (z_1^5, \sqrt{5} z_1^3 z_2 (z_1 {\overline w}_1 + z_2 {\overline
w}_2), \sqrt{5}
z_1 z_2^2  (z_1 {\overline w}_1 + z_2 {\overline w}_2)^2, z_2^5) \eqno 
(19) $$
From (19) we obtain the matrix $C({\overline w})$.
$$  \begin{pmatrix}   1 &  0 & 0 & 0
\cr 0 &  \sqrt{5} {\overline w_1} & 0 & 0
\cr 0 & \sqrt{5} {\overline w_2} &  \sqrt{5} {\overline w_1}^2 & 0
\cr 0 &  0 & 2 \sqrt{5} {\overline w_1}{\overline w_2} & 0
\cr 0 & 0 &  \sqrt{5} {\overline w_2}^2  & 0
\cr 0 & 0 & 0 & 1  \end{pmatrix} \eqno (20) $$
\end{example}

The degree of a rational proper map between balls equals the degree
of its numerator.  The map $C(w)$ is independent of $w$ if and only if
$f$ is homogeneous.

Suppose $f_t$ is a family, all of degree at most $d$, and some $f_t$
is of degree $d$.
If all $f_t$ have embedding dimension $N$, then $C$ is of size $K$ by $N$.

We can recover $f$ from $C$. Let $E_1,...,E_N$ be the usual basis for
${\bf C}^N$.
Then $C({\overline w})(E_k)$ is the $k$-th component of the numerator
of $f$, homogenized by writing $1 = \langle z, w \rangle$.
Then we can dehomogenize.

\begin{example} We recall Example 2.1. Put $t = {\rm cos}(\theta)= c$
and $s = {\rm sin}(\theta)$.
$$ H_t(z_1,z_2) = $$
$$ (cz_1 - sz_2^2, z_1 z_2, (cz_1-sz_2^2) (sz_1+cz_2^2),
z_1z_2(sz_1+cz_2^2), (sz_1+cz_2^2)^2). \eqno (21) $$
For each $t$, the map  $H_t$ has embedding dimension $5$.
When $t=1$, the degree drops, and hence the degree is not a
homotopy-invariant. \end{example}

To clarify this example, we compute the $X$-variety for the maps $H_t$.
Here is the matrix ${\overline C}(w)$:

$$  \begin{pmatrix}   c w_1^3  &  0 &  c s w_1^2 & 0 & s^2 w_1^2
\cr 3c w_1^2 w_2 & w_1^2  & 2 c s w_1 w_2  & s w_1 & 2 s^2 w_1 w_2
\cr 3c w_1 w_2^2 - s w_1^2  & 2 w_1 w_2 & cs w_2^2 + (c^2-s^2)w_1 &
sw_2 & 2 sc w_1 + s^2 w_2^2
\cr c w_2^3 - 2s w_1 w_2  &  w_2^2  & (c^2 - s^2) w_2  & c & 2sc w_2
\cr - s w_2^2 & 0 &  - s c & 0 & c^2  \end{pmatrix} $$

The determinant of this matrix is $c^2 w_1^6$. Thus, unless $c=0$, it
is generically invertible.
When $c \ne 0$, there is an exceptional fiber when $w_1 = 0$.
When $c=0$ we have a map of degree $3$; hence the rank cannot exceed $4$.

This example illustrates a general  method for constructing homotopies, which we 
elaborate in the last section.

\section{Whitney sequences}

Let $f:{\bf B}_n \to {\bf B}_N$ be a proper rational mapping. Let $A$
be a subspace of ${\bf C}^N$, and let $\pi_A$ denote orthogonal
projection onto $A$.
Following the first author's approach from [D], we may form the
new proper mapping $E_A (f)$, defined by

$$ E_A(f) = (\pi_A f \otimes z) \oplus (1- \pi_A)(f). $$

Suppose that $B$ is another subspace of ${\bf C}^N$ of the same
dimension $d$ as $A$, and $A \cap B = \{0\}$.
 Then there is a unitary mapping $U \in U(N)$ such that
$U(A) = B$. Since the unitary group is path connected, we can find a
one-parameter family
of unitary mappings connecting $U$ to the identity.
It follows that the maps $E_A(f)$ and $E_B(f)$ are homotopic in
dimension $K$, where $K =N + d(n-1)$. Example 2.1 is obtained via this
construction.

\begin{definition} A {\it Whitney sequence} is a collection $F_0, F_1,
...$ of rational proper maps from ${\bf B}_n$ to ${\bf B}_{N_k}$
defined as follows.
Put $F_0(z)=\phi_0$, where $\phi_0$ is an automorphism of ${\bf B}_n$.
Given $F_k : {\bf B}_n \to {\bf B}_{N_k}$,
let $A_k$ be a non-zero subspace of ${\bf C}^{n_k}$, and let $\pi_k$
denote orthogonal projection onto $A_k$.
Choose an automorphism $\phi_k$ of ${\bf B}_n$.
Choose a linear, norm-preserving 
injection $j_k$ to whatever target dimension we wish.
Define $F_{k+1}(z)$ by
$$ F_{k+1} = 
j_k \circ \big( (\pi_k F_k \otimes \phi_k) \oplus (1 - \pi_k) F_k \big). \eqno
(22) $$ \end{definition}

The degree of the rational function $F_k$ is at most $k+1$, but it can
be smaller. The following result
provides an analogue of the one-dimensional situation.

\begin{theorem} Let $\{F_k\}$ denote a Whitney sequence of proper
mappings. Each $F_k$ is homotopic to a monomial
proper mapping of degree $k+1$.
\end{theorem}

\begin{proof} The idea of the proof comes from Proposition 2.1. We
proceed by induction on the number of factors.
When $k=0$, the function $F_0$ is an automorphism, and hence homotopic
to the identity map (a monomial map of degree $1$) by Example 2.2.
Suppose for some $k$ that $F_k$ is homotopic to a monomial mapping
$G_k$ of degree $k+1$. Find a homotopy connecting
$\phi_k$ to the identity. Then $F_{k+1}$ is homotopic to the mapping

$$ G_{k+1} = (\pi_k G_k \otimes z) \oplus (1 - \pi_k) G_k. \eqno (23) $$

Note that $G_{k+1}$ if of degree $k+2$ if $\pi_k G_k$ is of degree
$k+1$. By the induction hypothesis
$G_k$ is of degree $k+1$. Since $A_k$ is not the trivial subspace, there is
a unitary map such that $\pi_k UG_k$ is also of degree $k+1$. Also
$(1- \pi_k)G_k$ has degree at most $k+1$.
Since the unitary group is path-connected, $UG_k$ is homotopic to
$G_k$. Hence $G_{k+1}$ is homotopic to a monomial mapping of degree
$k+2$.
\end{proof}

The mappings $H_t$ in Example 2.1 are each part of a Whitney sequence.
The degree is not a homotopy invariant
because the tensor products are taken on different subspaces, and
hence the tensor product need not increase the degree.

Not every proper rational mapping is a term of a Whitney sequence. For
example, even the monomial map
$(z,w) \to (z^3, \sqrt{3} zw, w^3)$ cannot be obtained in this fashion.
One must allow also the inverse operation of replacing $F_{k+1}$ in
(22) with $F_k$.

The following result indicates the significance of the target dimension
in the definition of homotopy.

\begin{theorem} Let $F_k:{\bf B}_n \to {\bf B}_N$ be a term in a Whitney sequence.
Then, $F_k \oplus 0$ is homotopic in target  dimension $N+1$ to the 
injection $z \to z \oplus 0$. 
\end{theorem}
\begin{proof}
We have already shown that $F_k$ is homotopy equivalent to a monomial mapping in target
dimension $N$. Furthermore this monomial mapping is in the image of the tensor product
construction. We claim that such monomial maps are always homotopically equivalent
to the identity in target dimension $N+1$. This conclusion is 
trivial for maps of degree $1$. For $d\ge 2$, consider an arbitrary 
monomial mapping $f$ of degree $d$, of embedding dimension $N$, and in the 
image of the tensor product 
operation. We will show that it is homotopic in target dimension $N+1$
to a monomial mapping of degree at most $d-1$, of embedding dimension at 
most $N$, and also in the image of the tensor product mapping.
 To do so, we order the monomials of degree $d$ lexicographically.
Choose the last monomial $m$ that occurs. Then there is a monomial $q$ of 
degree $d-1$ such that the $n$ monomials $z_1 q, \dots, z_n q$ include 
$m$. After renumbering we may assume that $m=z_nq$. For some polynomial 
map $g$ of degree at most $d$ we can write   
$$f= g  \oplus z_1q \oplus z_2 q \oplus \dots \oplus  z_n q. \eqno (24) $$
Now we replace the last $n$ of these components with $\lambda$ times 
each,
and we add an $(N+1)$-st component $\sqrt{1-\lambda^2} q$. For each 
$\lambda \in [0,1]$ the result is  a proper map to the $N+1$ ball. When 
$\lambda=0$ we obtain a map whose last component is $q$ and for which $m$
no longer appears. We continue in 
this fashion one monomial (of degree $d$) at a time, obtaining homotopies 
in target dimension $N+1$ that (in finitely many steps) eliminate all 
terms of degree $d$. Since homotopy is an equivalence relation, the 
composition defines 
a homotopy $H_t$  in target dimension $N+1$. Since $f$ 
is a term of a Whitney sequence, we are in the same 
situation as before, with $d$ lowered. Eventually we reach
a linear map and the result follows.  \end{proof}

More information on the tensor product operation appears, for example,
in [D] and [D3].

\section{bibliography}

[CS] J. A. Cima and T. J. Suffridge, Boundary behavior of rational
proper maps. {\it Duke Math. J.} 60 (1990), no. 1, 135-138

\medskip

[D] J. P. D'Angelo,  Several Complex Variables and the Geometry of
Real Hypersurfaces,
CRC Press, Boca Raton, Fla., 1992.

\medskip

[D1] J. P. D'Angelo, Homogenization, reflection, and the $X$-variety,
{\it Indiana Univ. Math J.} 52 (2003),
1113-1134.

\medskip

[D2] J. P. D'Angelo,
Proper holomorphic maps between balls of different dimensions.
{\it Michigan Math. J.} 35 (1988), no. 1, 83-90.

\medskip

[D3] J. D'Angelo, Proper holomorphic mappings,
positivity conditions, and isometric imbedding, {\it J. Korean Math
Society}, May 2003, 1-30.

\medskip

[DL1] J. P. D'Angelo and Jiri Lebl, On the complexity of proper
mappings between balls, {\it Complex Variables and Elliptic
Equations},
Volume 54, Issue 3, Holomorphic Mappings (2009), 187-204.

\medskip
[DL2] J. P.  D'Angelo  and J. Lebl, Complexity results for CR mappings
between spheres,
{\it Int. J. of Math.}, Vol. 20, No. 2 (2009),  149-166.

\medskip
[Fa]  James J. Faran, Maps from the two-ball to the 
three-ball, Invent. Math. 68 (1982), no. 3, 441-475.

\medskip
[F1] F. Forstneric,
Extending proper holomorphic maps of positive codimension,
{\it Inventiones Math.}, 95(1989), 31-62.

\medskip

[F2] F. Forstneric, Proper rational maps: A survey, Pp 297-363 in
{\it Several Complex Variables: Proceedings of the Mittag-Leffler
Institute, 1987-1988},  Mathematical Notes 38, Princeton Univ.
Press, Princeton, 1993.

\medskip
[H] Huang, X., On a linearity problem for proper maps between balls in
complex spaces
of different dimensions, {\it J. Diff. Geometry} 51 (1999), no 1, 13-33.

\medskip
[HJ] Huang, X., and Ji, S.,
Mapping ${\bf B}_n$ into ${\bf B}_{2n-1}$, {\it Invent. Math.} 145
(2001), 219-250. 

\medskip
[JZ] Shan Yu Ji and Yuan Zhang,
Classification of rational holomorphic maps from $B^2$ into $B^N$ with
degree $2$.
{\it Sci. China Ser. A } 52 (2009), 2647-2667.

\medskip

[L] J. Lebl, Normal forms, Hermitian operators, and CR maps of spheres
and hyperquadrics,
{\it  Michigan Math. J.} 60 (2011), no. 3, 603-628.

\medskip

[LP] J. Lebl and Peters, H., Polynomials constant on a hyperplane and
CR maps of spheres, {\it Illinois J. Math.} 56 (2012), no. 1, 155-175.

\medskip

[R] M. Reiter, Holomorphic Mappings of Hyperquadrics from ${\bf C}^2$ 
to ${\bf C}^3$, PhD dissertation, Univ. of Vienna, 2014.

\end{document}